\documentclass[12pt,final]{amsart}
\usepackage{amssymb,euscript}
\usepackage{amscd}
\usepackage[notref,notcite]{showkeys}
\usepackage{latexsym,todonotes}
\usepackage{epsfig,pxfonts}
\usepackage{amsfonts}

\usepackage{xr,xr-hyper}
\usepackage{bbm,ifpdf}
\ifpdf
  \usepackage{pdfsync,hyperref}
\fi
\usepackage[margin=3cm,footskip=25pt,headheight=21pt]{geometry}
\usepackage{fancyhdr,mathrsfs,nicefrac}
\newtheorem{theorem}{Theorem}[section]

\newtheorem{proposition}[theorem]{Proposition}
\newtheorem{corollary}[theorem]{Corollary}

\newtheorem{conjecture}[theorem]{Conjecture}

{
\theoremstyle{plain}
\newtheorem{definition}[theorem]{Definition}

\newtheorem{remark}[theorem]{Remark}
}

\newcommand{\nc}{\newcommand}
\nc{\cat}{\mathcal{V}}

\newcommand{\arxiv}[1]{\href{http://arxiv.org/abs/#1}{\tt arXiv:\nolinkurl{#1}}}
\newcommand{\Spec}{\operatorname{Spec}}

\renewcommand{\dim}{\operatorname{dim}}

\nc{\Coh}{\mathsf{Coh}}
\newcommand{\Sym}{\operatorname{Sym}}

\newcounter{subeqn}

\makeatletter
\@addtoreset{subeqn}{equation}

\newcommand{\K}{\mathbf{k}}

\nc{\eE}{\EuScript{E}}
\nc{\eF}{\EuScript{F}}

\newcommand{\Z}{\mathbb{Z}}
\newcommand{\Q}{\mathbb{Q}}

\newcommand{\supp}{\operatorname{supp}}

\newcommand{\Fq}{\mathbb{F}_q}
\newcommand{\C}{\mathbb{C}}

\nc{\ep}{\epsilon}

\newcommand{\bd}{\partial}
\newcommand{\Bd}{\mathbf{d}}

\newcommand{\la}{\leftarrow}

\newcommand{\M}{\mathfrak{M}}

\nc{\Bv}{\mathbf{v}}
  \nc{\Bw}{\mathbf{w}}
\nc{\coho}{\EuScript{G}}
\nc{\sllhat}{\mathfrak{\widehat{sl}}_\ell}
\nc{\slehat}{\mathfrak{\widehat{sl}}_e}
\nc{\glehat}{\mathfrak{\widehat{gl}}_e}

\newcommand{\secs}{\Gamma_\bS}

\nc{\cH}{\mathcal{H}}
\nc{\HC}{\mathbf{HC}}
\nc{\fL}{\mathfrak{L}}

\renewcommand{\la}{\lambda}

\nc{\Tr}{\operatorname{Tr}}
\nc{\tU}{\mathcal{U}}

\newcommand{\al}{\alpha}
\renewcommand{\b}{\beta}

\newcommand{\Hom}{\operatorname{Hom}}

\renewcommand{\bd}{{\mathbf{d}}}

\nc{\lift}{\gamma}

\newcommand{\Ob}{\operatorname{Ob}}
\newcommand{\cO}{\mathcal{O}}

\newcommand{\Ext}{\operatorname{Ext}}

\newcommand{\cD}{\mathcal{D}}
\newcommand{\bS}{\mathbb{S}}

\newcommand{\cQ}{\mathcal{Q}}

\newcommand{\cM}{\mathcal{M}}

\newcommand{\cN}{\mathcal{N}}

\newcommand{\Loc}{\operatorname{Loc}}

\newcommand{\excise}[1]{}

\nc{\wela}{\EuScript{X}}
\nc{\rola}{\EuScript{Y}}

\newcommand{\End}{\operatorname{End}}

\newcommand{\fM}{\mathfrak{M}}
\newcommand{\fN}{\mathfrak{N}}

\newcommand{\fZ}{\mathfrak{Z}}

\newcommand{\mmod}{\operatorname{-mod}}

\newcommand{\thetitle}{Geometry and categorification}

\begin{document}

\renewcommand{\theitheorem}{\Alph{itheorem}}
\usetikzlibrary{decorations.pathreplacing,backgrounds,decorations.markings}
\tikzset{wei/.style={draw=red,double=red!40!white,double distance=1.5pt,thin}}

\noindent {\Large \bf 
\thetitle}
\bigskip\\
{\bf Ben Webster}\footnote{Supported by the NSF under Grant
  DMS-1151473 and the Alfred P. Sloan Foundation}\\
Department of Mathematics,  University of Virginia, Charlottesville, VA
\bigskip\\
{\small
\begin{quote}
\noindent {\em Abstract.} 
We describe a number of geometric contexts where categorification
appears naturally: coherent sheaves, constructible sheaves and sheaves
of modules over quantizations.  In each case, we discuss how ``index
formulas'' allow us to easily perform categorical calculations, and readily
relate classical constructions of geometric representation theory to
categorical ones.
\end{quote}
} 
\medskip 

\section{Introduction}
\label{sec:introduction}

``Categorification'' is a very flexible concept.  It simply refers to
the idea that it can be very interesting to take a set and add morphisms
between its objects.  Its very flexibility means that it is an idea
which must be employed carefully.  It has not proven very effective to
start with a simple algebraic object and to hunt aimlessly for
categorifications of it.  It is much more reliable to have a
``machine'' which produces categories for you in a way that gives you
some hope of understanding how they decategorify.

Thus, geometry is a natural context for categorification because it is a
natural source of categories.  The categories that appear in geometry
also have a natural geometric toolkit for producing functors (using
push-pull along correspondences) and calculation (using index
formulas).  Both of these can be more difficult to understand in other approaches to
categorification, such as algebraic or diagrammatic.  The focus of
this paper will be on describing some of the basic ways of applying
geometry to construct categories, how the underlying geometry can help
with understanding these categories, and how to apply these principles
in some of the most illuminating special cases.

Categorification can also help us to better understand geometry.  The categories which
appear naturally in this context shed light on the nature of the
spaces they are connected to.
Often, the full structure of a category like coherent sheaves on an
algebraic variety is simply ``too rich for our blood,'' an amount of
information that exceeds our ability to take it in.
Decategorification allows us cut out much of the extraneous
complication and understand some of the structure of this category,
and thus learn something about the underlying space.  

This paper is structured around 3 different geometric contexts which
naturally lead to categorification:
\begin{itemize}
\item In Section \ref{sec:k-theory}, we consider categories of coherent sheaves and
  algebraic K-theory.  This is arguably the first place in the
  literature where the modern philosophy of categorification appears,
  and the most likely to be somewhat familiar to the general reader.
\item In Section \ref{sec:funct-sheaf-corr}, we consider categories of
  constructible sheaves and the function-sheaf correspondence.  While
  perhaps a more specialized taste, this is actually an incredibly
  powerful theory, with connections to deep number theory.  In this
  author's opinion, it is one any aspiring categorifier should know a bit of.
\item In Section \ref{sec:sympl-resol}, we consider categories of
  sheaves of modules over quantizations.  This is the least familiar
  context, and one still under development.  Unlike the other two
  examples, we have not had the benefit of having Grothendieck around
  to help us with it. However, progress on it has been made, which we
  will briefly discuss here.
\end{itemize}

\section{K-theory}
\label{sec:k-theory}

The first appearance of categorification in its modern form was in
topological K-theory.  Given a topological space $X$, we can consider
the additive category of vector bundles on $X$.  To better understand
this category we consider its {\bf Grothendieck
group} $K(X)$.  This is the abelian group generated by symbols $[A]$ for $A$ a
vector bundle, subject to the relation
\begin{equation}\label{eq:oplus}
[A]+[B]=[A\oplus B].
\end{equation}
Note, most authors consider vector bundles up to stable equivalence,
which is the same as considering the kernel of the map $K(X) \to \Z$
sending a vector bundle to its rank.  Of course, this construction has
many variations where we consider bundles with additional structure.
See the classic books of Atiyah \cite{AtiyahK} or Karoubi \cite{KarK}
for more details on K-theory.

\subsection{Coherent sheaves}
\label{sec:coherent-sheaves}

This topological introduction is perhaps a little misleading.  We'll
instead be working with algebraic varieties. Both a curse and a blessing of this
geometric approach is that almost every construction that appears has
several variations that make sense in different contexts.  Throughout,
I'll usually work in whatever context is most convenient for me.
Thus, I could consider holomorphic vector bundles on a complex
manifold or locally free coherent sheaves on a scheme or variety
(over the complex numbers or some other field).  For a projective
variety over $\C$, these notions are the same by \cite{GAGA}, so the
reader will not lose much by thinking about whichever one they prefer.

Now, we will more seriously study the category $\Coh(X)$ of coherent
sheaves on a scheme $X$.  Readers who are less happy with algebraic
geometry might also think about the category of modules over a
commutative ring $R$, which is the same as $\Coh(\Spec R)$, the
coherent sheaves on the spectrum $\Spec R$ of $R$.
\begin{definition}
  Attached to this category, we have two natural Grothendieck groups:
  \begin{itemize}
  \item Let $K_0(X):=K_0(\Coh(X))$ be the formal span of $[\mathcal{F}]$
    for all {\bf locally free}\footnote{Those of you thinking about
      $R$-modules should restrict to projective $R$-modules.} coherent
    sheaves $\mathcal{F}$ on $X$, modulo the relation
    \begin{equation}
      [\mathcal{E}]-[\mathcal{F}]+[\mathcal{G}]=0\label{eq:ses}
    \end{equation}
    for any short exact sequence
    \begin{equation}\label{eq:EFG}
      0\longrightarrow \mathcal{E} \longrightarrow
      \mathcal{F}\longrightarrow \mathcal{G}\longrightarrow 0,
    \end{equation} which is a slight modification of the
    relation \eqref{eq:oplus},
    and 
  \item Let $G_0(X) :=G_0(\Coh(X))$ be the span of $[\mathcal{F}]$ for all
    coherent sheaves modulo the same relation \eqref{eq:ses}.
  \end{itemize}
\end{definition}

This
modified relation is needed since neither of these categories is
semi-simple, while every short exact sequence of topological vector
bundles splits.

 Note that  there's an obvious homomorphism $K_0(X)\to G_0(X)$, but this need
  not be an isomorphism. If $X$ is quasi-projective and smooth, then it induces
  a natural isomorphism $K_0(X)\cong G_0(X)$ since every sheaf has a
  finite length locally free resolution\footnote{In fact, by a theorem
  of Serre, a quasi-projective variety is smooth if and only if every
  coherent sheaf has a finite length resolution.}.  On the other hand,
this will not hold in many other cases.

\begin{remark} If $R=\C[t]/(t^2), X=\Spec R$,
  then $K_0(X)$ is spanned by the class of the regular module
  $R$, and $G_0(X)$ is spanned by the class of the 1-dimensional
  module $\C\cong R/tR\cong tR$.  The short exact sequence
  \begin{equation*}
    0\longrightarrow tR \longrightarrow
    R\longrightarrow R/tR\longrightarrow 0,
  \end{equation*}
  shows that under the isomorphism $G_0(X)\cong \Z$, the subgroup
  $K_0(X)$ is sent to $2\Z$.  
 The difference between this case and the smooth case is that the minimal projective resolution of $R/tR$ is
  the infinite complex
  \[\cdots \longrightarrow R\overset{t} \longrightarrow R\overset{t}
  \longrightarrow\cdots \overset{t} \longrightarrow R \longrightarrow
  R/tR.\]
\end{remark}

Note that $K_0(X)$ is a ring, with multiplication given by $[\mathcal{E}][
\mathcal{F}]=[\mathcal{E} \otimes_{\mathcal{O}_X}
\mathcal{F}]$.  We cannot endow $G_0(X)$ with a compatible ring
structure in general since $\otimes_{\mathcal{O}_X}$ is not exact. As
we'll discuss below, we can sometimes fix non-exactness by considering
higher derived functors, but our example above shows that this can't
work here: if $X=\Spec \C[t]/(t^2)$ then $K_0(X)\cong \Z$ as a ring, so 
$G_0(X)$ can only be a ring if $\nicefrac{1}{2}\Z$ is.  The problem is
that $\operatorname{Tor}^i_R(\C,\C)\cong \C$ for all $i$, so we would have to use
the divergent series $\nicefrac{1}{2}=1-1+1-1+\cdots$.

\subsection{Functoriality}
\label{sec:functoriality}

Our first task is to understand how maps between varieties induce
functors between categories and thus maps between K-groups.  Of
course, the desired functors are pushforward and pullback of coherent
sheaves, as defined in \cite[II.5]{Hartshorne}. Unfortunately, interpreted naively, neither of these functors
is exact.  Pushforward $f_*$ is left exact, and thus has right derived
functors $\mathbf{R}^if_*$ (see \cite[III.8]{Hartshorne}), and pullback $f^*$ is right exact and thus
has left derived functors $\mathbf{L}^if^*$.  These have the usual
long exact sequences
\begin{equation}
\cdots \longrightarrow\mathbf{R}^{i-1}f_*\mathcal{G}\longrightarrow\mathbf{R}^if_*\mathcal{E} \longrightarrow
\mathbf{R}^if_*\mathcal{F}\longrightarrow \mathbf{R}^if_*\mathcal{G}
\longrightarrow\mathbf{R}^if_*\mathcal{E}\longrightarrow\mathbf{R}^{i+1}f_*\mathcal{E}  \longrightarrow
\cdots\label{eq:R-LES}
\end{equation}
\begin{equation}
\cdots \longrightarrow\mathbf{L}^{i+1}f^*\mathcal{G}\longrightarrow\mathbf{L}^if^*\mathcal{E} \longrightarrow
\mathbf{L}^if^*\mathcal{F}\longrightarrow \mathbf{L}^if^*\mathcal{G}
\longrightarrow\mathbf{L}^if^*\mathcal{E}\longrightarrow\mathbf{L}^{i-1}f^*\mathcal{E}  \longrightarrow
\cdots\label{eq:L-LES}
\end{equation}
which show that the maps 
\begin{equation}
  \label{eq:push-pull}
  [\mathcal{F}]\mapsto
  \sum_{i=0}^{\infty}(-1)^i[\mathbf{R}^if_*\mathcal{F}]\qquad [\mathcal{G}]\mapsto
  \sum_{i=0}^{\infty}(-1)^i[\mathbf{L}^if^*\mathcal{G}]
\end{equation}
are compatible with the relation \eqref{eq:ses} whenever these
infinite sums make sense. As often happens, the two different maps
above make sense for the two different versions of the Grothendieck
group:
\begin{proposition}
  For a projective morphism $f\colon X\to Y$, the formulas of \eqref{eq:push-pull} define maps 
\begin{equation}
  \label{eq:push-pull-fun}
 G_0(f_*)\colon G_0(X)\to G_0(Y)\qquad  K_0(f^*)\colon K_0(Y)\to K_0(X)
\end{equation}
\end{proposition}
\begin{proof}
  For any coherent sheaf $\mathcal{F}$ on $X$, the sheaves $\mathbf{R}^if_*\mathcal{F}$ are coherent by
 \cite[III.8.8(b)]{Hartshorne}, and vanish for $i>\dim X$ by \cite[III.2.7]{Hartshorne}.
  Thus, the first sum of \eqref{eq:push-pull} is finite and
  well-defined.  Note that we could not do this for $K_0$ since
  $\mathbf{R}^if_*\mathcal{F}$ might not be locally free even if
  $\mathcal{F}$ is.  

For a locally free sheaf $\mathcal{G}$ on $Y$, the pullback
$f^*\mathcal{G}$ is locally free and the higher pullbacks
$\mathbf{L}^if^*\mathcal{G}$ for $i>0$ are 0, so the second sum of
\eqref{eq:push-pull} is well-defined (and in fact only has 1 non-zero
term).  
\end{proof}
The map $\Spec \C\to \Spec R$ induced by the unique ring homomorphism
$R\to \C$ illustrates that it's impossible to define these maps on the
``wrong'' Grothendieck group:
\begin{itemize}
\item the regular module $\C$ is free, but its pushforward $f_*\C$ is
  just the module $R/tR$, which doesn't have a corresponding class in
  $K_0$.
\item the pullback $\mathbf{L}^if^*(R/tR)\cong \C$ for all $i$, so the
  second sum in \eqref{eq:push-pull} doesn't converge.  
\end{itemize}
Of course, this map is somewhat pathological, and there are conditions
one can impose that will guarantee this maps make sense.  In
particular:
\begin{itemize}
\item If the source and target are smooth and quasi-projective,
  then we can freely switch between $K_0$ and $G_0$ and so $G_0(f^*)$
  and $K_0(f_*)$ make sense for a projective (or more generally,
  proper) morphism.
\item If $f$ is flat, then $f^*$ is exact, and $G_0(f^*)$ is
  well-defined. 
\end{itemize}

\subsection{Chern character and index formulas}
\label{sec:chern-char-index}

Thus, we can supply ourselves with a great number of categories and
exciting functors between them.  But as we said in the introduction,
this is particularly powerful because we can understand the
Grothendieck group and calculate the behavior of these maps in
geometric terms.  Assume from now on that $X$ is smooth and
projective over $\C$.
\begin{theorem}
  There is a unique homomorphism
  $\operatorname{ch}\colon K_0(X)\cong G_0(X)\to H^*(X;\Q)$ called
  {\bf Chern character} which is:
  \begin{enumerate}
  \item compatible with pullback
    \[\operatorname{ch} (f^*\mathcal{F})=f^*\operatorname{ch}
    (\mathcal{F})\]
   \item  sends the class $[\mathsf{L}]$ of a line bundle $\mathsf{L}$ to \[\exp(c_1( \mathsf{L}))=1+c_1(\mathsf{L})+\frac{c_1(\mathsf{L})^2}{2}+\cdots\]
    where $c_1(\mathsf{L})$ is the first Chern class, the cohomology
    class dual to the divisor defined by a meromorphic section of
    $\mathsf{L}$.
  \end{enumerate}
\end{theorem}
  This definition is well-defined because of the splitting principle:
  given a vector bundle $E\to X$ of rank $n$, we can consider the flag
  space $\mathsf{Fl}(E)$ given by pairs of a point $x\in X$ and a
  complete flag $V_1\subset V_2\subset \cdots \subset V_n=E_x$.  The
  cohomology ring is given by
  \[H^*(\mathsf{Fl}(E);\Q)\cong
  H^*(X;\Q)[\al_1,\dots,\al_n]/(e_k(\boldsymbol{\alpha})=c_k(E));\]
  here the classes $\alpha_i$ are the first Chern classes of the
  induced vector bundles $V_i/V_{i-1}$.  We can use
  $e_k(\boldsymbol{\alpha})=c_k(E)$ as a definition of the Chern
  classes, and $\operatorname{ch}
    (E)=\sum_{i=1}^ne^{\al_i}$.  Note that this implicitly gives a
    complicated but concrete formula for $\operatorname{ch}$ in terms
    of Chern classes, since each homogeneous part is a
symmetric polynomial in Chern roots $\al_i$, and thus a polynomial in
the Chern classes. We can also define the Todd class $\operatorname{td}(E)=\prod_{i=1}^n
\frac{\al_i}{1-e^{-\al_i}}$.

\begin{remark}
Note that even if $X$ is defined over a field other than $\C$, we can
still define this homomorphism, with the target given by the Chow
ring $A^*(X;\Q)$, the ring spanned by subvarieties of $X$ modulo rational
equivalence.  
\end{remark}
The Chern character is certainly not an
isomorphism, but is not so far from being one either:
\begin{proposition}
  The kernel of $\operatorname{ch}\otimes\, \Q\colon K_0(X)\otimes \Q\to
  H^*(X;\Q)$ is spanned by elements $[\mathcal{F}]-[\mathcal{G}]$
  where we have an isomorphism of underlying topological $\C$-vector
  bundles $\mathcal{F}\cong_{\C}\mathcal{G}$.  
\end{proposition}
The inelegance of this theorem comes from the fact that we are
relating two ``incompatible'' structures.  If we consider the
topological K-theory of $X$ as a manifold, or replace $H^*(X;\Q)$ by
the Chow ring $A^*(X;\Q)$, then this map will be an isomorphism.
For many nice varieties, we have an isomorphism $H^*(X;\Q)\cong
A^*(X;\Q)$, so the map above becomes an isomorphism.  For example,
this is the case for any variety which has an affine paving.

While there are many things to be said about K-theory, perhaps the
most important for us is the first example of an index formula: the
Grothendieck-Hirzebruch-Riemann-Roch theorem.   
\begin{theorem}[Grothendieck, Hirzebruch; \mbox{\cite[\S 7]{BorSer}}]
Let $X,Y$ be smooth and projective over $\C$.  Then for any map
$f\colon X\to Y$ and any locally free sheaf $\mathcal{F}$, we have
that
\[ \operatorname{ch}(f_*\mathcal{F})
\operatorname{td}(T_Y)=f_*(\operatorname{ch}(\mathcal{F})
\operatorname{td}(T_X)).\]
In particular, in the case where $Y=\Spec \C$ is a point, we have that 
\[\sum_{i=0}^\infty(-1)^i\dim  H^i(X;\mathcal{F})=\int_X \operatorname{ch}(\mathcal{F})
\operatorname{td}(T_X).\]
\end{theorem}

Since this is a particularly focus of categorifiers, let us note that
using internal $\Hom$ allows us to compute the Euler form on $K_0(X)$,
which is defined by 
\[\langle
[\mathcal{E}],[\mathcal{F}]\rangle=\sum_{i=0}^\infty\Ext^i(\mathcal{E},\mathcal{F}).\]
This sum is well-defined because we have that 
\[\Ext^i(\mathcal{E},\mathcal{F})=H^i(X;\mathcal{E}^\vee\otimes
\mathcal{F}),\] which vanishes for $i>\dim X$.  This also allows us to
see that \[\langle[\mathcal{E}],[\mathcal{F}]\rangle=\int_X \operatorname{ch}(\mathcal{F})\operatorname{ch}(\mathcal{E}^\vee)
\operatorname{td}(T_X).\]
\begin{subsection}{The Weyl character formula}
  One of the most important categorification problems is understanding
  the characters of representations of groups.  For complex simple Lie
  groups, this problem is solved by the Weyl character formula.  While
  there are many proofs of this beautiful formula, one of the most
  remarkable is obtained by combining GHRR as above with the
  Borel-Weil-Bott theorem, which shows that for each highest weight
  $\la$, there is a line bundle $\mathcal{L}_\la$ on the flag variety
  $X:=G/B$ of $G$ such that
  $H^0(X;\mathcal{L}_\la)\cong V_\la$ and $H^i(X;\mathcal{L}_\la)=0$
  for $i>0$.  Thus, we have that:
\[\dim V_\la=\int_X \operatorname{ch}(\mathcal{L}_\la)
\operatorname{td}(T_X).\]  
Consider the symmetric algebra $\Sym^\bullet(\mathfrak{t}^*_\Q)$  in the weights of
the maximal torus $T$, and let $\epsilon(f)$ be the constant term of $f\in \Sym^\bullet(\mathfrak{t}^*_\Q)$.
We can identify $H^*(X;\Q)$ with the
symmetric algebra $\Sym^\bullet(\mathfrak{t}^*_\Q)$ modulo $\Sym^\bullet(\mathfrak{t}^*_\Q)^W_+$,
the polynomials with $\epsilon(f)=0$ which are symmetric under the action of the Weyl
group; in this realization, the operation of integration is given by 
\begin{equation}
\int_Xp=\epsilon\Bigg(\frac{\sum_{w\in W}(-1)^ww\cdot p}{\prod_{\alpha\in
    \Delta^+}\alpha}\Bigg).\label{eq:int}
\end{equation}
The positive roots $\al\in  \Delta^+$ are the Chern roots of the
tangent bundle, so we have $\operatorname{td}(T_X)=\prod_{\al\in
  \Delta^+}\frac{\al}{1-e^{-\al}}$, and the Chern character of the line
bundle is 
$\operatorname{ch}(\mathcal{L}_\la)=e^{\la}$, so \[\operatorname{ch}(\mathcal{L}_\la)
\operatorname{td}(T_X) =\frac{e^{\la+\rho}\prod_{\al\in
  \Delta^+}\al}{\prod_{\al\in \Delta^+}e^{\al/2}-e^{-\al/2}.}\] Thus, we have that 
\begin{align}
\dim V_\la&=\int_X \operatorname{ch}(\mathcal{L}_\la)
\operatorname{td}(T_X) \notag \\
&=\epsilon\Bigg(\frac{1}{\prod_{\alpha\in \label{eq:WCF}
    \Delta^+}\alpha}\sum_{w\in W}(-1)^w \frac{e^{w(\la+\rho)}\prod_{\al\in
  \Delta^+}\al}{\prod_{\al\in
  \Delta^+}e^{\al/2}-e^{-\al/2}}\Bigg)\\
&=\epsilon\Bigg(\frac{\sum_{w\in W}(-1)^w
e^{w(\la+\rho)}}{\prod_{\al\in
  \Delta^+}e^{\al/2}-e^{-\al/2}}\Bigg).\notag
\end{align}

If we evaluate the RHS using L'H\^opital's rule,  we obtain the Weyl dimension formula:
\[\dim V_\la=\prod_{\al\in \Delta^+}\frac{\langle \la+\rho,\al\rangle}{\langle \rho,\al\rangle}.\]

We can also think $T$-equivariantly for the natural
torus action on $X$.  In
this case, the equation \begin{equation*}
\int_Xp=\frac{\sum_{w\in W}(-1)^ww\cdot p}{\prod_{\alpha\in
    \Delta^+}\alpha},\label{eq:eqint}
\end{equation*} 
gives the integral in equivariant cohomology, valued in
$H_T(\operatorname{pt})\cong {\Sym}^\bullet(\mathfrak{t}^*_\Q)$. We can then interpret \eqref{eq:WCF} in the
completion of this
$T$-equivariant cohomology:
\begin{equation*}
\int_X \operatorname{ch}(\mathcal{L}_\la)
\operatorname{td}(T_X)=\frac{1}{\prod_{\alpha\in
    \Delta^+}\alpha}\sum_{w\in W}(-1)^w \frac{e^{w(\la+\rho)}\prod_{\al\in
  \Delta^+}\al}{\prod_{\al\in
  \Delta^+}e^{\al/2}-e^{-\al/2}}=\frac{\sum_{w\in W}(-1)^w
e^{w(\la+\rho)}}{\prod_{\al\in
  \Delta^+}e^{\al/2}-e^{-\al/2}}.\label{eq:eqWCF}
\end{equation*}
The
Chern character of $V_\la$, considered as a $T$-equivariant coherent
sheaf on a point, is the sum over weight spaces $\sum_{\mu}\dim
(V_\la)_\mu e^\mu$, so from equivariant GHRR, we obtain the usual Weyl character formula:
\[\sum_{\mu}\dim
(V_\la)_\mu e^\mu=\frac{\sum_{w\in W}(-1)^w
e^{w(\la+\rho)}}{\prod_{\al\in
  \Delta^+}e^{\al/2}-e^{-\al/2}}.\]
\end{subsection}

\section{The function-sheaf correspondence}
\label{sec:funct-sheaf-corr}

\subsection{Euler characteristic}

\nc{\Con}{\mathsf{Sh}}

Instead of coherent sheaves, we can also consider constructible
sheaves. Let $k$ be a commutative ring, and for any topological space
$X$, we let $k_X$ denote the sheaf of locally
  constant\footnote{The presheaf that assigns the constant $k$-valued
    functions to any open subset is not a sheaf, since if a subset
    $U=U_1\cup U_2$ is the union of two open subsets with $U_1\cap
    U_2=\emptyset$, then, a function that takes value $a$ on $U_1$ and
  $b$ on $U_2$ must define a section of the sheafification by the
  gluing property.  The locally constant functions are, essentially by
definition, the sheafification of this presheaf.} $k$-valued functions on $U_i$.
\begin{definition}
We call a sheaf of $k$-modules
  $\mathscr{F}$ on a topological space $X$ a {\bf local system} if it is
  locally constant, i.e. if there is a finite open cover $\{U_i\}$,
  such that $\mathcal{F}|_{U_i}\cong k_{U_i}^{\oplus m}$ for some
  integer $m$. 
\end{definition}
A great example of a local
system is the flat sections of a vector bundle $E$ with a flat connection on
a compact manifold.
While globally there may be no sections, there is always an open cover
$U_i$ where for any $u\in U_i$, each element of the fiber $E_u$ 
extends uniquely to a covariantly constant section of $E$ over $U_i$.
Thus a local system is a sheaf where nearby fibers are isomorphic, but
in order to make this identification canonical, we have to shrink to a
smaller neighborhood (and the identification may depend on the
neighborhood we choose).

Now, assume that the topological space we consider is a complex
quasi-projective variety.  This space has two natural topologies both
induced from the embedding in projective space: the {\bf Zariski
  topology}, whose generating open sets are the locus where some
meromorphic function on projective space has neither a zero nor a
pole, and the {\bf classical topology} which is induced by the usual
smooth manifold structure on $\mathbb{CP}^n$.  Constructible sheaves are sheaves of finitely generated $k$-modules on
algebraic varieties that are locally modeled on
local systems in the {\it classical} topology.   We can define them inductively
by saying that a sheaf $\mathscr{F}$ on a variety of dimension $n$ is
{\bf constructible}
if there is a Zariski open subset $U$ such that $\mathscr{F}|_U$ is a
local system in the classical topology and on the
complementary Zariski closed subset $V$ (of lower dimension), $\mathscr{F}|_V$
is constructible\footnote{Note the odd mix of the classical and
  Zariski topologies here; this will be eased a bit when we consider
  the \`etale topology.}.   We let $\Con(X)$ be the category of
constructible sheaves on $X$, and $D^b\Con(X)$ the subcategory of the
bounded derived category of all sheaves of vector spaces where all
complexes have constructible cohomology.

Certain aspects of constructible life are actually much simpler than
coherent sheaves.  For example, instead of an index formula valued in
homology we obtain one valued in the 
space of constructible functions on $X$; as with sheaves, we
inductively define {\bf constructible functions} on a variety of dimension
$n$ to be those constant on a Zariski open subset $U$, with the restriction
to the complement constructible.  Let $C(X)$ be the ring of
constructible functions on $X$ (with pointwise addition and
multiplication).  The map from sheaves to functions is
one version of the function-sheaf correspondence:
\[\phi_{\mathscr{F}}(x)=\dim  \mathscr{F}_x.\] More generally, for a
complex of sheaves with constructible cohomology, we take 
\[\phi_{\mathscr{F}}(x)=\sum_{i\in \Z}(-1)^i\dim  \mathcal{H}^i(\mathscr{F}_x).\]
Here, we use $\mathcal{H}$ to emphasize we are just taking cohomology of
a complex, not any kind of sheaf cohomology.  
The long exact sequence (of sheaves) induced by a short exact sequence
(of complexes of sheaves) shows that this map factors through the
Grothendieck group $G_0(\Con(X))$, that this: \begin{equation}
\phi_{\mathcal{E}}-\phi_{\mathcal{F}}+\phi_{\mathcal{G}}=0\label{eq:ses2}
\end{equation} whenever we have a short exact sequence of complexes like \eqref{eq:EFG}.  Furthermore, 
note 
that $\phi_{f^{-1}\mathscr{F}} =f^*\phi_{\mathscr{F}}$, where $f^*$ is
the usual pullback of functions.   However, for constructible sheaves,
there are two kinds of natural pushforward: the usual pushforward
$f_*\mathscr{F}(U)=\mathscr{F}(f^{-1}(U))$ and the pushforward with
proper supports:
\[f_!\mathscr{F}(U)=\{s\in \mathscr{F}(f^{-1}(U))|
\operatorname{supp}(s)\to U\text{ is proper}\}.\] While the former is
probably more familiar, it is actually more convenient for us to use
the latter.

Just like coherent sheaves, constructible sheaves have an index
 formula.  

\begin{definition}
  For a constructible function $\phi$ on $X$ and a map $f\colon X\to
  Y$, we let $f_!$ be the unique linear map such that:
  \begin{itemize}
  \item If $X=U\sqcup V$ with $U$ open and $V$ closed, we have
    $f_!\phi=f_!\phi|_U+f_!\phi|_V$.
\item If $\phi=1$ is the constant function, then
  $f_!1(y)=\chi_c(f^{-1}(y))$ is the compactly-supported Euler
  characteristic $\chi_c(X)=\sum_{i=0}^\infty (-1)^i\dim H^i_c(X)$ of the fiber
  over each point.
  \end{itemize}
The pushforward in the case where $Y=*$ is a point is sometimes called
Euler integration and denoted $\int \phi\, d\chi$.  Pushforward is
simply performing this integration over fibers.  
\end{definition}
Note, this is very close to the functor defined in \cite{McChern}, but
using compactly supported Euler characteristics.  This pushforward
is a small modification of that defined in \cite[\S 1.4]{GrinMc}, by
using the standard trick of factoring $f_!$ into the extension by zero into
a compactification, and then usual pushforward.

With this convention, we can now show:
\begin{theorem}
  $\phi_{f_!\mathscr{F}}=f_!\phi_{\mathscr{F}}$ 
\end{theorem}
\begin{proof}
First, let us prove this in the case of an open inclusion $U\overset{i}\hookrightarrow X$.  In this
case, the value at a point in $u$ is obviously unchanged.  At a point
in $X\setminus U$, the stalk of $f_!\phi_{\mathscr{F}}$ is $0$, so the
value of $\phi_{f_!\mathscr{F}}$ at the point is 0, as is true of
$f_!\phi_{\mathscr{F}}$. 

Let us prove this by induction on the dimension of the source variety $X$.
  Note that if we have a decomposition into open and closed subsets
  $U\overset{i}\hookrightarrow X \overset{j}\hookleftarrow V$, then we
  have a short exact sequence \[0\longrightarrow i_!i^*\mathscr{F}\longrightarrow \mathscr{F} \longrightarrow
  j_!j^*\mathscr{F}\to 0.\]
Thus, we have 
\[\phi_{f_!\mathscr{F}}-f_!\phi_{\mathscr{F}}=\phi_{f_!
  i_!i^*\mathscr{F}}+\phi_{f_!
  j_!j^*\mathscr{F}}-(fi)_!\phi_{i^*\mathscr{F}}-(fj)_!\phi_{j^*\mathscr{F}}.\]
Since $\dim V <\dim X$, we have $\phi_{f_!
  j_!j^*\mathscr{F}}=(fj)_!\phi_{j^*\mathscr{F}}$.  Thus, it suffices
to check the result after removing an arbitrary closed subvariety from
$X$.  Thus, we can assume that $\mathscr{F}$ is a local system
since this holds on an open subset.  Taking an open cover of $X$ (in
the {\it classical} topology) that
trivializes $\mathscr{F}$, and applying the Mayer-Vietoris spectral
sequence associated to this cover, we see that we get the same answer
for any local system, and thus can consider $\mathscr{F} =k_{X}$.  In this case, the stalk of
$R^if_!k_{X}$ at a point $y\in Y$ is given by $H_c^i(f^{-1}(y))$ by
base change\footnote{Note the importance of using compactly supported
  pushforward here, since the usual pushforward does not have this
  property for non-proper maps.} \cite[4.5.4]{SGA4andahalf}, so
indeed, the result follows from the equality $f_!1(y)=\chi_c(f^{-1}(y))$.
\end{proof}

\subsection{\'Etale cohomology}
\label{sec:arithmetic}

There is a more refined version of this theorem, which is {\it much}
more difficult to prove; it was the endpoint of 3 decades of
remarkable work in algebraic geometry.  

Computing the Euler characteristic of a complex algebraic variety is a
crude analogue of counting the number of points in an algebraic
variety over a finite field (it behaves a lot like the number of
points when $p=1$).  In particular, both quantities are invariant
under scissors congruence.  

Like Euler characteristic, the number of points in an
algebraic variety has a cohomological interpretation, which is again
an index formula.
Understanding this interpretation correctly requires a lot of
difficult technical details, but these are surprisingly easy to bypass
to understand the general framework.  Those looking for more details
should look first at the notes of Milne \cite{milneLEC}, which cover
all the basic ideas while remaining relatively accessible.  If still
more details are sought, the reader can turn to the earlier book of
Milne \cite{MilneEC} or that of Kiehl and Weissauer \cite{KW}, both in
English, or earlier French sources, such as \cite{SGA4andahalf}.

The first scary-sounding thing is the {\bf \'etale topology} on an
algebraic variety $X$ (see \cite[\S 2-7]{milneLEC} for general
discussion).  This is not a topology in the usual sense, but a
Grothendieck topology: its ``open subsets'' are given by maps
$\nu\colon U\to X$ (see, for example, \cite[\S 4]{milneLEC}).  You can
think of this as imposing a topology where certain maps are formally
locally invertible (even if there's no underlying map of spaces that
really achieves this).  The \'etale topology on a smooth manifold is
the topology where the open sets are manifolds $U$ equipped with a
smooth map $\nu\colon U\to X$ such that at each point $u\in U$, the
differential $T_u\nu\colon T_uU\to T_{\nu(u)}X$ is an isomorphism (so
this map is both an immersion and a submersion, a property we call
{\bf \'etale}).  You can easily work out that on a manifold, the
\'etale topology is equivalent to the usual one, by the inverse
function theorem\footnote{Two Grothendieck topologies are equivalent
  if for any ``neighborhood'' $\nu\colon U\to X$ in one topology,
  there is a neighborhood $\eta:V\to X$ in the other topology such
  that $\nu$ factors through $\eta$ (so $\eta$ ``contains'' $\nu$) and
  another $\eta':V'\to X$ such that $\eta'$ factors through $\nu$ (so
  $\nu$ ``contains'' $\eta'$).  An open subset in the usual topology
  includes via an \'etale map, and for any \'etale map, the inverse
  function theorem guarantees there's a neighborhood $V'$ of $u$ that
  maps diffeomorphically to an open subset of $X$.}.  Thus the \'etale
topology is essentially the topology where we declare {\it a priori}
that the inverse function theorem is true.  Note that while this is
not a ``real'' topology, we can still make sense of sheaves in this
topology, and define cohomology $H^*_{\text{\'et}}(X;\Lambda)$ in any
abelian group $\Lambda$, and compactly supported cohomology
$H^*_{\text{\'et},c}(X;\Lambda)$ using \v Cech
cohomology\footnote{It's more ``morally correct'' to define this
  cohomology using derived functors; however, for reasonable schemes,
  this is the same by \cite[10.2]{milneLEC}.}. 

For
manifolds, this yields nothing new, since the inverse function theorem
really is true.  However, we can apply the same trick in situations where it
is not, like the Zariski topology on an algebraic variety.    This
genuinely changes the topological behavior of this variety.  For
example, it's a well-known fact that the cohomology of a complex
algebraic variety in the Zariski topology is trivial, whereas if we
compute it in the \'etale topology with coefficients in a finite
abelian group $\Lambda$, then, by \cite[21.1]{milneLEC}, it coincides
with the usual Betti cohomology of the complex points (in the
classical topology):
\[H^*_{\text{\'et}}(X;\Lambda)\cong H^*(X(\mathbb{C};\Lambda)).\]
The fact that a finite group is required here is a minor nuisance; it
is mostly one of the technicalities I suggest the reader ignore, but
it does mean that typically, we work with coefficient groups and rings
which are built from finite ones.  Thus, by definition, we let 
\[H^*_{\text{\'et}}(X;\Z_\ell):=\varprojlim H^*_{\text{\'et}}(X;\Z/\ell^n\Z)\]
\[H^*_{\text{\'et}}(X;\Q_\ell):=
H^*_{\text{\'et}}(X;\Z_\ell)\otimes_{\Z_\ell}\Q_\ell.\]
for any prime $\ell$.  Unfortunately, these do not give the same result
as computing ``directly'' in the \'etale topology.  The same
comparison theorems to Betti cohomology exist for these groups by the
universal coefficient theorem.

There are other differences between the \'etale and Zariski topologies
which are relevant for local systems.  For
example, on $\C^*$, consider the local system defined by solutions of
$x\frac{df}{dx}=\frac{1}{2}f$.  Solutions to this are provided by the
branches of the square root function. Thus, it is very easy to find an
open subset in the classical topology where this local system is
trivialized by removing a single ray; however, there is no Zariski
open subset (that is, the complement of finitely many points) where
this local system is trivialized.  On the other hand, in the \'etale
topology, it is trivialized by the ``neighborhood'' $\C^*\overset
{x\mapsto x^2}\longrightarrow\C^*$.  

In the previous section, we defined the constructible sheaves using a
funny mix of the classical and Zariski topologies, essentially because
there aren't enough interesting local systems in the Zariski topology;
now knowing about the \'etale topology, we might prefer to consider
\'etale constructible sheaves\footnote{Note that certain classical
  local systems cannot be trivialized in the \'etale topology.
  Solutions to $x\frac{df}{dx}=\alpha f$ only will be if $\alpha$ is
  rational.  This is yet another complication it will probably not
  greatly benefit the reader to cogitate upon.}.  

The remarkable thing about the \'etale topology on a complex algebraic
variety is that it is a purely algebraic object: the tangent spaces
and differential 
can be rephrased in algebraic language, so we can speak of a map
between schemes being \'etale and thus define \'etale
neighborhoods and the \'etale topology for an arbitrary scheme.  Thus,
we can define cohomology groups $H^*_{\text{\'et}}(X;\Lambda)$ for an
arbitrary scheme.

This results in a second remarkable comparison theorem:  assume that
we have a scheme $X$ defined over $\Z$ (for example, a projective variety
defined by polynomials with integral coefficients)\footnote{This is
  really a much stronger hypothesis than we need.  With a bit more
  work, this theory can be made to work for any variety over the
  complex numbers.}.  We can consider
the base-change of this variety $\C$ or to the algebraic closure $\bar{\mathbb{F}}_p$
of any finite field of characteristic $p$.
\begin{theorem}[\mbox{\cite[20.5 \& 21.1]{milneLEC}}]
  If $\ell$ is any prime distinct from $p$, then we have
  \[H^*_{\text{\'et}}(X\otimes \bar{\mathbb{F}}_p;\Z_\ell)\cong
  H^*_{\text{\'et}}(X\otimes \C;\Z_\ell)\cong H^*(X(\C);\Z_\ell).\]
\end{theorem}

Since the former group is defined purely using the characteristic $p$
geometry of $X\otimes \bar{\mathbb{F}}_p$ (that is, the solutions to
our polynomials over finite fields) and the latter purely using the
topology of the complex solutions, this is a pretty remarkable
theorem.  

\subsection{The Grothendieck trace formula}
\label{sec:groth-trace-form}

However, one might wonder what purpose it serves in relation to
categorification.  These results about cohomology are in fact proven
in a categorical context.  We can consider \'etale local systems and
constructible sheaves not
just on complex algebraic varieties, but on any scheme, in particular
one of characteristic $p$.  Just like on complex algebraic varieties,
these sheaves are endowed with pushforward and pullback functors.  

But rather than just considering Euler characteristic of stalks, we have a much
richer invariant, which incorporates the action of the Frobenius.  The
Frobenius of interest to us is the relative Frobenius
$\operatorname{Fr}\colon X\otimes \bar{\mathbb{F}}_p\to X\otimes
\bar{\mathbb{F}}_p$ which is induced by raising functions on $X\otimes
\mathbb{F}_p$ to the $p$th power.  For a projective variety, this is
the map of raising the projective coordinates to the $p$th power; note
that this is an automorphism of the variety since the polynomials have
integer coefficients.  For any constructible sheaf $\mathcal{F}$ on $X\otimes
\mathbb{F}_p$, we have a canonical isomorphism $\mathcal{F}\cong
\operatorname{Fr}_*\mathcal{F}$.  This means that if $x\in X$ is an
$\mathbb{F}_p$-rational point (i.e. one whose coordinates lie in
$\mathbb{F}_p$), then, we have an induced Frobenius map
$\operatorname{Fr}\colon \mathcal{F}_x\to \mathcal{F}_x$.  If we have
that $\mathcal{F}$
is a complex of sheaves with $\Q_\ell$-constructible cohomology defined on  $X\otimes
\mathbb{F}_p$, then we get an action on the stalks of the cohomology
on this point.  This action respects the differentials of the long exact sequence on
cohomology, so we find that:
\begin{proposition}
  The function \[\Phi(\mathcal{F})(x)=\sum_{i\in
    \Z}(-1)^i\operatorname{tr}(\operatorname{Fr}\mid
  \mathcal{H}^i(\mathscr{F}_x))\] defines a map from the Grothendieck
  group $K(\mathsf{Sh}(X\otimes \mathbb{F}_p))$ to $\Q_\ell$-valued functions on $X(\mathbb{F}_p)$.
\end{proposition}
More generally, we can consider the $n$th power of the Frobenius map,
and define
\[\Phi^{(n)}(\mathcal{F})(x)=\sum_{i\in
    \Z}(-1)^i\operatorname{tr}(\operatorname{Fr}^n\mid
  \mathcal{H}^i(\mathscr{F}_x))\]
for $x\in X(\mathbb{F}_{p^n})$.  Compared to just taking Euler
characteristic, this map is much more powerful. 
\begin{theorem}\label{th:injective}
  The map from the Grothendieck group to $\prod_{n=1}^\infty
  \Q_\ell[X(\mathbb{F}_{p^n})]$ defined by $\prod_{n=1}^\infty\Phi^{(n)}$ is injective.
\end{theorem}
These functions also satisfy a trace formula, remarkably (or maybe not
so remarkably?) also due in large part to Grothendieck.  Compatibly
with our notation before, if $f\colon S\to T$ is a map of finite sets,
and $\Phi\colon S\to k$ is a map to any ring $k$ (or more generally
abelian group), then $f_!\Phi(t)=\sum_{s\in f^{-1}(t)}\Phi(s)$.  
\begin{theorem}\label{th:Grothendieck}
  For any map between $\mathbb{F}_p$-schemes $f\colon X\to Y$, we have
  $\Phi^{(n)}(f_!\mathcal{F})=f_!\Phi^{(n)}(\mathcal{F})$, where on
  the RHS, we use $f$ to denote the induced map
  $X(\mathbb{F}_{p^n})\to Y(\mathbb{F}_{p^n})$.  In particular, if
  $f\colon X\to \Spec \mathbb{F}_p$, and $\mathcal{F}=\Q_\ell$, then
  we find that 
\[\#X(\mathbb{F}_{p^n})=\sum_{i=0}^\infty(-1)^i\operatorname{tr}(\operatorname{Fr}^n\mid
  {H}^i _{\text{\'et},c} (X;\Q_{\ell})).\]
\end{theorem}
While it might seem strange, this theorem is thoroughly topological in
nature: it's simply the Lefschetz fixed point theorem applied to the
Frobenius.  

Note, this means that the eigenvalues of Frobenius have a powerful
effect on the number of points in $X(\mathbb{F}_{p^n})$ as we change
$n$.  For example, Wiles's proof of Fermat's last theorem proceeded by
showing that a counter-example would lead to the existence of an
elliptic curve whose Frobenius eigenvalues are too strange to actually
exist.

\subsection{Grassmannians and $\mathfrak{sl}_2$}
\label{sec:grassm-mathfr}

\nc{\Gr}{\mathsf{Gr}}

Now, let's actually apply these theorems a bit.  One very interesting
and relevant example is given by the system of Grassmannians and
partial flag varieties.  Let $\Gr(r,n)$ be the Grassmannian defined
over $\Z$; you can either think of this as the projective variety
defined by the Plucker relations (which have integer coefficients) or
as the variety whose functor of points sends a ring $k$ to the set of
module quotients $k^n\to V$ such that $V\cong k^{n-r}$.  If $k$ is a field,
this is the collection of $r$-dimensional subspaces in $k^n$.  

Given $r<r'\leq n$, we have a partial flag variety $\Gr(r\subset
r',n)$, given by pairs of subspaces with one inside the other.  These
have their own Pl\"ucker relations, also defined over $\Z$, and
natural maps
\[\Gr(r,n)\overset{\pi_{r,r'}}\longleftarrow \Gr(r\subset
r',n) \overset{\pi_{r',r}}\longrightarrow \Gr(r',n).\]

There are functors 
\[\mathcal{E}=(\pi_{r,r+1})_*\pi_{r+1,r}^*[r]\qquad \mathcal{F}=(\pi_{r,r-1})_*\pi_{r-1,r}^*[n-r]\]
relating the categories $\mathcal{C}_r:=\Con(\Gr(r,n))$ for all
$0\leq r\leq n$.  The brackets indicate homological shift in the
derived category, but also require changing the action of the
Frobenius in order to keep the mixed structure pure of weight 0 by a
factor of $p^{-r/2}$; the square of this operation is called ``Tate twist.''  The
overall effect is that $\Phi^{(m)}(\mathcal{G}[r])=p^{-rm/2}\Phi^{(m)}(\mathcal{G})$.

 These functors are biadjoint up to shift since
$\pi_{r',r} $ is smooth and proper.  We can understand the action of these functors using the
index formulas we've defined (Theorem \ref{th:Grothendieck}).  For two
subspaces $V,V'$ in a larger vector space, we 
write $V\overset{1}\subset V'$ if $V\subset V'$ and $\dim(V'/V)=1$.  Let $q=p^m$. 
\begin{proposition} For any constructible function $G$ on $\Gr(r,n)$
  and $V\in \Gr(r,n)$, we have 
  \[\displaystyle \Phi^{(m)}(\mathcal{E}G)(V)=\sum_{V\overset{1}\subset V'}q^{-r/2}\Phi^{(m)}(G)(V')\]
  \[\displaystyle \Phi^{(m)}(\mathcal{F}G)(V)=\sum_{V'\overset{1}\subset V}q^{(r-n)/2}\Phi^{(m)}(G)(V')\]
\end{proposition}
Now, consider the difference
\begin{align*}
\Phi^{(m)}(\mathcal{E}\mathcal{F}G)(V)-\Phi^{(m)}(\mathcal{F}\mathcal{E}G)(V)&=\sum_{V\overset{1}\subset
  V'}q^{(-n+1)/2}\Phi^{(m)}(G)(V)-\sum_{V'\overset{1}\subset
  V}q^{(-n+1)/2}\Phi^{(m)}(G)(V)\\ &=q^{\nicefrac{(-n+1)}2}(1+\cdots
+q^{n-r-1}-1-\cdots -q^{r-1})\Phi^{(m)}(G)(V) \\ &=q^{\nicefrac{(-n+1)}2}\frac{q^{n-r}-q^r}{q-1}\Phi^{(m)}(G)(V) \\ &=\frac{q^{\nicefrac{(n-2r)}{2}}-q^{\nicefrac{(2r-n)}{2}}}{q^{\nicefrac 12}-q^{-\nicefrac 12}}\Phi^{(m)}(G)(V)
\end{align*}

By the joint injectivity of the maps $\Phi^{(m)}$, this implies that 
\begin{align}\label{EFp}
\mathcal{E}\mathcal{F}G&\cong
\mathcal{F}\mathcal{E}G\oplus G[n-2r-1]\oplus
\cdots \oplus G[2r-n+1]  &(n-2r\geq 0)\\
\label{EFm}
\mathcal{E}\mathcal{F}G&\oplus G[2r-n-1]\oplus
\cdots \oplus G[n-2r+1]\cong
\mathcal{F}\mathcal{E}G &(n-2r\leq 0)
\end{align}

This is a categorified version of the $U_q(\mathfrak{sl}_2)$ relation
\[EF-FE=\frac{K-K^{-1}}{q-q^{-1}}.\]
In fact the functors $\mathcal{E}$ and $\mathcal{F}$ define a
categorical action of $\mathfrak{sl}_2$, in the sense of Chuang and
Rouquier.
\begin{definition}
  A collection of categories $\mathcal{C}_n$ and functors 
\[\mathcal{E}
  \colon \mathcal{C}_n \to \mathcal{C}_{n+2} \qquad \mathcal{F}
  \colon \mathcal{C}_n \to \mathcal{C}_{n-2}\] form a categorical
  $\mathfrak{sl}_2$-action if
  \begin{enumerate}
  \item the functors $\mathcal{E}$ and $\mathcal{F}$ are biadjoint, and
  \item satisfy the relations (\ref{EFp}--\ref{EFm}) up to
    isomorphism, and 
  \item there is an action of the nilHecke algebra on the $m$-fold composition
    $\mathcal{E}^m$.
  \end{enumerate}
\end{definition}

The first 2 of these properties have geometric interpretations we've
discussed: they arise from the properness and smoothness of the maps,
and from the point counting above.  The third also has a geometric
interpretation.  The functor $\mathcal{E}^m$ is a push-pull functor
for the correspondence \[\Gr(r\subset r+1\subset \cdots \subset r+m,n)=\{V_r\overset{1}\subset
V_{r+1}\overset{1}\subset\cdots \overset{1}\subset V_{r+m}\subset
k^n\}\] over $\Gr(r,n)$ and $\Gr(r+m,n)$.  This fits into the diagram:
\[\tikz[->,very thick]{
\matrix[row sep=12mm,column sep=20mm,ampersand replacement=\&]{
\&\node (a) {$\Gr(r\subset r+1\subset \cdots \subset r+m,n)$}; \&;\\
\node (b) {$\Gr(r,n)$}; \& \node (c) {$\Gr(r\subset r+m,n)$};\&\node (d) {$\Gr(r+m,n)$};\\
};
\draw (c) -- node[above,midway]{$\pi_{r,r+m}$}(b);
\draw (a) -- node[right,midway]{$q$} (c);
\draw (c)-- node[above,midway]{$\pi_{r+m,r}$} (d);
}\]
The vertical map $q$ is a fiber bundle with fiber given by the
complete flag variety $\mathsf{Fl}(m)$ on an $m$-dimensional space; the functor
$\mathcal{E}^m$ can be rewritten as convolution with the pushforward
of the constant sheaf on $\Gr(r\subset r+1\subset \cdots \subset r+m,n)$, which is isomorphic to
$H^*(\mathsf{Fl}(m))$ tensored with the constant sheaf on
$\Gr(r\subset r+m,n)$.  Thus, the action of the nilHecke algebra on
this functor is inherited from its action on the cohomology of the
flag variety.  This is a more sheafy interpretation of the results of
\cite{Laucq,LauSL2} which phrases the same action in terms of the
cohomology of Grassmannians and these correspondence, which arise when
we take hypercohomology of the sheaves discussed above.

\begin{remark}
  This action can be extended to other partial flag varieties,
  obtaining a categorification of $\mathfrak{sl}_n$ when we consider
  $n$-step flags.  This corresponds to the calculations of Khovanov
  and Lauda in \cite{KLIII}.  This action is discussed in greater
  detail in \cite{Webcomparison}.  
\end{remark}

\subsection{Flag varieties}
\label{sec:flag-varieties-1}

The structure of constructible sheaves on flag varieties is a deep and
beautiful subject.  To stay within the bounds on this paper, we will
mainly concentrate on the relationship to the Hecke algebra.  Consider
a Coxeter group $W$ generated by the set $S$ with the relations
\[s^2=e\qquad (st)^{m(s,t)}=e\]
\begin{definition}
  Let $\mathcal{H}_v(W)$ be the algebra over $\Z[v,v^{-1}]$ generated
  by $T_s$ for $s\in S$ with the relations:
\[ (T_s+1)(T_s-v)=0 \qquad T_sT_t\cdots =T_tT_s\cdots.\]
where the latter two products both have $m(s,t)$ terms.
\end{definition}

On the other hand, we can consider the $\Fq$-points of a split simple algebraic group $G$ over
$\Fq$ (for example, $PGL_n(\Fq),Sp_{2n}(\Fq), SO_m(\Fq), \cdots$), and
let $B$ be the $\Fq$-points of a Borel.  This is just the group
elements which preserve an appropriate flag (which must be self-dual for symplectic
or orthogonal groups).  We have a Bruhat decomposition which gives a
bijection between the Weyl group $W$ and the double cosets $BwB$ for
$B$ in $G$.  The Weyl group of $G$ is a Coxeter group (for example,
for $PGL_n(\Fq)$, it is the symmetric group $S_n$).  

The set of functions on the double cosets $B\backslash G/B$ has a
natural multiplication:  let 
\[f\star g(BwB)=\sum_{\substack{Bw'\in B\backslash G\\w''B\in G/B
    \\BwB=Bw'w''B}}f(Bw'B)g(Bw''B).\]
This arises naturally when we identify $\K[B\backslash G/B]$ with
$\End_G(\K[G/B])$, acting by the same formula.  We can also identify
$B\backslash G/B$ with the set of $G$-orbits on $G/B\times G/B$ via
the map $(g_1B,g_2B)\mapsto Bg_1^{-1}g_2B$.  In this realization, we
can write this multiplication as
\[f\star g =(\pi_{13})_!(\pi_{12}^*f\cdot \pi_{23}^*g)\] where
$\pi_{ij}\colon (G/B)^3\to (G/B)^2$ is the map forgetting the factor
which is not listed.

\begin{theorem}[\cite{Iwahori}]
  The set of functions on $B\backslash G/B$ is isomorphic to the
  specialization $H_q(W)$ at $v=q$.  
\end{theorem}

Just to give a taste of how this map works, let's consider the case of
$G=PGL_n(\Fq)$.  In this case, $G$ acts transitively on the set of
complete flags in $\Fq^n$, with the stabilizer of each flag being a
Borel.  Thus $G/B$ is just the set of complete flags in this vector
space, with choosing a Borel $B$ specifying a preferred flag
$V_1\subset V_2\subset \cdots \subset V_n=\Fq^n$.  The identity in the
convolution multiplication is the indicator function of this preferred
flag $V_\bullet$; as a function on $G/B\times G/B$, this becomes the
the indicator of the diagonal.  Our generators of the Hecke algebra are  $T_{(i,i+1)}$ for $(i,i+1)$
one of the adjacent  transpositions in the symmetric group $S_n$.  We
send $T_{(i,i+1)}$ to the indicator function of the Schubert cell
\[C_{(i,i+1)}=\{V'_\bullet|V'_i\neq V_i, V_j'=V_j \text{ for all }j\neq i\}.\]
as a function on $G/B\times G/B$, this becomes the
the indicator of the pairs of flags that differ only in
$i$-dimensional subspace.

From our earlier discussion, we know that the natural way to
categorify functions on $G/B$ is to consider constructible sheaves on
$G/B$.  In order to obtain $B$-invariant functions, we need to
consider $B$-equivariant sheaves, or equivalently $G$-equivariant
sheaves on $G/B\times G/B$.
\begin{definition}
  We call a sheaf on a $G$-scheme $X$  {\bf weakly $G$-equivariant} if there
  is an isomorphism of sheaves $a^*\mathcal{F}=p^*\mathcal{F}$ where
  $a,p\colon G\times X \to X$ are the action and projection maps
  $(g,x)\mapsto gx,x$ respectively.\footnote{The reader might rightly
    protest that it would be better to consider strong equivariant sheaves,
    which is are sheaves together with a choice of isomorphism which
    satisfies a cocycle (associativity) condition.  This is fair, but
    not necessary at the present.  Ultimately, if the reader wants to ``do this right,''
they should use equivariant derived categories as well
\cite{BL,WWequ}, but in this context, this is not necessary for
understanding the basic point.}

We let $D^b(\mathsf{Sh}_G(X))$ denote the full subcategory of the
derived category whose cohomology is weakly equivariant and constructible.
\end{definition}
Now, let us consider the underlying algebraic group $G$, and $G/B$ as
an algebraic variety over $\Fq$.  We'll be interested in the derived category
$D^b(\mathsf{Sh}_G(G/B\times G/B))$ of weakly $G$-equivariant sheaves.
The $G$-orbits of $G/B\times G/B$ are all simply connected and there
are finitely many of them, so you can think of these as trivial vector
bundles on the different orbits, with some sort of topological glue
holding them together.  

In particular, if we take the function $\Phi$ for one of these
sheaves, we obtain a function on $G/B\times G/B$ which is constant on
$G$-orbits.  Put differently, $\Phi$ defines a natural map from the Grothendieck group
$D^b(\mathsf{Sh}_G(G/B\times G/B))$ to the Hecke algebra.  This map is
surjective, since we can consider the extension $i_!(\Q_\ell)_Y$ from
an orbit $Y$, which hits the indicator function of the orbit.

This map is not injective, but this is only because we only considered
a single field. Thus, given one sheaf $\mathcal{F}$ on $G/B$, we have
that $\Phi(\mathcal{F}^{\oplus q})= \Phi(\mathcal{F}(1))=q \Phi(\mathcal{F})$ where
$\mathcal{F}(1)$ is the Tate twist of $\mathcal{F}$.  However, 
\[\Phi^{(n)}(\mathcal{F}^{\oplus q})= q
\Phi^{(n)}(\mathcal{F})\qquad \Phi^{(n)}(\mathcal{F}(1))=q^n
\Phi^{(n)}(\mathcal{F})\]
so considering larger $n$ will fix this problem.  If we do this
carefully, we can construct a category whose Grothendieck group is
$\mathcal{H}_v(W)$: the subcategory of $D^b(\mathsf{Sh}_G(G/B\times
G/B))$ where the Frobenius acts by elements of $q^\Z$ on every stalk,
and $v$ corresponds to Tate twist.

Thus we find that:
\begin{corollary}
  The category $D^b\mathsf{Sh}_G(G/B\times G/B)$  categorifies the
  Hecke algebra.
\end{corollary}
  In fact, the
category $D^b\mathsf{Sh}_G(G/B\times G/B)$ is a geometric realization of the categorification
of the Hecke algebra by Soergel bimodules, or the diagrammatics of
Elias and Williamson (see \cite[Th. 6]{Webcomparison} and its proof).

This is only a very small taste of a very large story; this is
discussed in much greater detail in the book of Hotta, Takeuchi and
Tanisaki \cite{HTT}.  The category
$\mathsf{Sh}_G(G/B\times G/B)$ is related to very interesting
categories of representations of the Lie algebra $\mathfrak{g}$ called
Harish-Chandra bimodules and category $\cO$.  

One topic we did not have the space to consider is that of perverse
sheaves and intersection cohomology.  This is a huge topic, and
provides an important lens for all of the examples with constructible
sheaves we've considered.  One good starting point is the introductory
article \cite{WhatisPS}.  One of the key consequences of this theory
is that the category
$D^b(\mathsf{Sh}_G(G/B\times G/B))$ contains special objects called
{\bf intersection cohomology sheaves}.  These special complexes of 
sheaves give a special basis of the Hecke algebra called the
Kazhdan-Lusztig basis.  This is one of the prototypical examples of a
special (or ``canonical'') basis arising from categorification.  

\subsection{Hall algebras}
\label{sec:hall-algebras}

Hall algebras are discussed in much greater detail in the survey
article of Schiffmann \cite{SchiffHall}.  Here, we will consider the
rather narrow topic of how they fit in with the function-sheaf
correspondence; the general experience of the author is that this
connection is not nearly as well-known as it should be.  

Hall algebras arise from our philosophy when the underlying space $X$
is the moduli space of objects $\Ob_{\mathcal{A}}$ in some abelian
category.  Making careful mathematical sense of such a space can be
quite tricky; usually it must be thought of as some kind of stack.
Thus, throughout, we'll deal mostly with the most familiar example: if
$\Gamma$ is a quiver, then there is an abelian category of
representations of $\Gamma$, that is, of representations of $\Gamma$'s
path algebra.

There is a geometric space whose points are the
isomorphism classes of representations of $\Gamma$. For a dimension
vector $\mathbf{d}\colon \Gamma\to \Z_{\geq 0}$, let
\[ E_{\Bd}(k)=\oplus_{i\to j}\Hom(k^{d_i},k^{d_j})\] for $k$ any commutative ring;
this defines the functor of points for an algebraic variety over
$\Z$.  This space has an action of the affine algebraic group $G_\Bd=\prod
GL_{d_i}(k)$, where $GL_{d_i}(k)$ act by precomposition on
arrows from $i$ and by postcomposition on arrows to $i$.  

You can think of $E_\Bd$ as being a quiver representation on a free
$k$-module with a choice of basis in each of the spaces, and $G_\Bd$
as the group that acts by changing bases.  Thus, two elements of
$E_{\Bd}(k)$ represent the same representation of $\Gamma$ if and only
if they are in the same orbit under $G_\Bd(k)$.
\begin{definition}
  We let the moduli space of representations of this quiver be the
  union of the quotient spaces\footnote{In this case, it's much
    harder to sweep the issue of equivariance under the rug.  However,
    we are nothing if not persistent, and thus will do our best to
    achieve said sweeping.} $E_{\Bd}/G_{\Bd}$ for all $\Bd$.
\end{definition}

This moduli space has an additional structure, which arises from the
fact that there are short exact sequences, which form a related moduli space.

For the quiver $\Gamma$, you can break these sequences up into components where the submodule has dimension $\bd'$ and
the quotient $\bd''$. You can think of this as the space 
\[ E_{\Bd',\Bd''}(k)=\oplus_{i\to j}\Hom(k^{d_i'},k^{d_j'})\oplus
\Hom(k^{d_i''},k^{d_j'})\oplus \Hom(k^{d_i''},k^{d_j''})\]
modulo the action of the group
\[G_{\Bd',\Bd''}:=\prod_i\{g\in GL(k^{d_i'}\oplus k^{d_i''})|g(k^{d_i'})=k^{d_i'}\}.\]

The moduli space of short exact sequences has 3 projection maps $\pi_s,\pi_t,\pi_q$ considering the
submodule, total module and quotient.  In the case of quiver
representations, this is considering the action
on $k^{d_i'},k^{d_i'}\oplus k^{d_i''}, $ and $ k^{d_i''}$
respectively.

\begin{definition}
  We let the {\bf Hall algebra} be the constructible functions on
  $\Ob_{\mathcal{A}}(k)$ the points of the moduli space for $k=\C,\Fq$
  equipped the algebra structure
  \begin{equation}
    f\star g=(\pi_t)_*(\pi_s^*f\cdot \pi_q^*g);\label{eq:hall}.
  \end{equation}
\end{definition}

Similarly, we can define a monoidal structure on the category of
constructible sheaves on $\Ob_{\mathcal{A}}$ via essentially the same
formula:
\begin{equation}
\mathcal{F}\star\mathcal{G}=(\pi_t)_*(\pi_s^*\mathcal{F}\otimes
\pi_q^*\mathcal{G}).\label{eq:conv}
\end{equation}
\begin{remark}
  Actually all of these formulas are slightly wrong.  The equation \eqref{eq:hall} should have
  included a power of $q$ (where $q$ is the size of the field or $-1$
  if $k=\C$), depending on the dimension of the fiber of $\pi_s\times
  \pi_q$.  Similarly, \eqref{eq:conv} should include a homological
  shift by the same dimension.  This shift assures that convolution
  commutes with Verdier duality.  
\end{remark}

Thus Theorem \ref{th:Grothendieck} implies that \eqref{eq:hall}
categorifies \eqref{eq:conv}:
\begin{theorem}\label{Hall-categorify}
  The map $\Phi^{(n)}\colon
  K(\mathsf{Sh} (\Ob_{\mathcal{A}}(\mathbb{F}_{q})))\to
  \Q_\ell[\Ob_{\mathcal{A}}(\mathbb{F}_{q^n})]$ is a homomorphism of
  rings and for every class in
  $K(\mathsf{Sh}(\Ob_{\mathcal{A}}(\mathbb{F}_{q})))$, there is a $n$ for
  which $\Phi^{(n)}$ does not kill this class. 
\end{theorem}

In the case of representations of $\Gamma$, we have a second way of
thinking about this Hall algebra.  Work of Ringel \cite{Ringel} shows
that we have a homomorphism from the quantized
universal enveloping algebra $U_q(\mathfrak{n})$ of the maximal
unipotent subalgebra of the  associated Kac-Moody algebra of
$\Gamma$ to the Hall algebra.  This map sends the Chevalley generator $E_i$ to
the indicator function of a trivial representation which is 1
dimensional and supported on $i$.  Thus, Theorem \ref{Hall-categorify}
suggests we should be able to categorify $U_q(\mathfrak{n})$  by
replacing these indicator functions with the corresponding constant
sheaf on $E_{\al_i}/G_{\al_i}\cong */\mathbb{G}_m$.  This was, in fact,
carried out by Lusztig \cite{Lus91,Lusbook} and leads to his
construction of canonical bases for universal enveloping algebras.

This construction was given a different spin in the work of Rouquier
\cite{RouQH,Rou2KM} and Varagnolo-Vasserot \cite{VV}, who showed that
the resulting categories of sheaves can be understood algebraically
using KLR algebras.

\section{Symplectic resolutions}
\label{sec:sympl-resol}

There is one final context for categorification we want to discuss:
that of conical symplectic resolutions of singularities.  These are closely
allied to the constructible sheaves we discussed, but also bear some
similarities to coherent sheaves.

A {\bf conical symplectic resolution} is an algebraic variety $\fM$
over $\C$ which is equipped with:
\begin{itemize}
\item a birational projective morphism $\fM\to \fN$ to an affine
  variety
\item an algebraic 2-form $\Omega$ such that $\Omega$ is symplectic
\item a conical $\bS\cong \C^*$-action on $\fM$ and $\fN$ compatible with the map such
  that $\Omega$ has weight $n>0$ for this $\bS$-action.  
\end{itemize}
These varieties have many remarkable properties.  The most relevant
for us is that they can be quantized.  The symplectic form $\Omega$
induces a Poisson bracket on functions; this Poisson bracket is
actually the leading order part of a quantization.  That is:
\begin{theorem}
  There is a $\bS$-equivariant sheaf of algebras $\cQ$ flat over $\C[[h]]$ on $\fM$ such that
  $\cQ/h\cQ\cong \cO_\fM$, and $[f,g]\equiv h\{f,g\}\pmod h^2$.  In
  fact, up to isomorphism, such algebras are in canonical bijection
  with $H^2(\fM;\C)$.  
\end{theorem}
Since $[-,-]$ has weight $0$ under $\bS$, and $\{-,-\}$ weight $-n$,
we must have that $h$ has $\bS$-weight $n$ for the desired equation to hold.

The corresponding cohomology class is called the {\bf period} of a
quantization.  If we take the global sections $\Gamma(\fM;\cQ)$, then
we can ``set $h=1$'' by adjoining $h^{-1/n}$, and considering the
invariant sections.  Since $h^{-1/n}$ has weight $-1$, any section
which is a $\bS$-weight vector can be multiplied by an appropriate
power of $h$ to make it invariant.  Let $\cD:=\cQ[h^{-1/n}]$.
\begin{theorem}
  The invariant sections $A:=\Gamma(\cM;\cD)^\bS$ is a non-commutative
  algebra filtered by its intersections $A(m):=A\cap
  \Gamma(\fM;h^{-m/n}\cQ)$.  There is a canonical isomorphism
  $\C[\fM]\cong A$.  
\end{theorem}

We call a $\bS$-equivariant module over $\cD$ {\bf good} if it is
isomorphic to the base extension of a coherent (i.e. locally finitely
generated) $\cQ$ module.  The categories of good $\cD$ modules and
finitely generated $A$-modules are related by an adjoint pair of
functors $$\secs\colon\cD\mmod\to A\mmod\qquad \Loc\colon
A\mmod\to \cD\mmod$$ which are often, but not always equivalences.

\begin{remark} The category $\cD\mmod$ is much less sensitive to changing the
  parameter that defines the quantization.  Any line bundle $L$ on
  $\fM$ can be quantized to a right module $\mathcal{L}$ over $\cD$;
  however, this quantization does not have a natural left $\cD$-module
  structure.  Instead $\End_{\cD}(\mathcal{L})$ is a new quantization,
  corresponding to a cohomology class differing from that of $\cD$ by
  the Euler class of $L$.  

Since every element of $H^2(\fM;\Z)$ is the
  Euler class of a line bundle, this means that the
  $H^2(\fM;\Z)$-cosets of $H^2(\fM;\C)$ give Morita equivalent
  quantizations.  Thus, without loss of generality, we can add a large
  integer multiple of an ample class on $\fM$ to the period of $\cD$.
  For a sufficiently large multiple, we have that localization holds
  by \cite{BLPWquant}.
\end{remark}

The category $\cD\mmod$ also has the enormous advantage of allowing
one to use the methods of geometric categorification.  Kashiwara and
Schapira \cite{KSdq} define a notion of a Euler class for a good $\cD$-module.
While this has a general definition using Hochschild homology, it can
actually be thought of in a relatively straightforward way in this
special case.

The most interesting $\cD$-modules are {\bf holonomic}; that is, their
support is a half-dimensional subvariety of $\fM$.  If we fix a good
$\cD$-module sheaf $\cM$, and $x\in \supp \cM$ a point where the
support is smooth, then the completion of $\cM$ at this point will be
a free module over the functions on the completion of $\supp \cM$ at
this point.  The rank of this module is constant on an open subset of
the component containing $x$, and thus defines an invariant $r_C$ of this
component.  

\begin{definition}
  The Euler class of $\cM$ is the sum  $e(\cM)=\sum r_C[C]\in
  H_{\operatorname{top}}^{BM}(\supp\cM)$.  
\end{definition}
This is shown in \cite[7.3.5]{KSdq} to be equivalent to the more
general definition.  

Unfortunately, one has to be careful about what an index formula means
in this case.  There is no definition of a pushforward in this
context.  Instead, one has to rely on the interactions between
modules.  For example:
\begin{theorem}
  Let $\cM$ and $\cN$ be good holonomic modules such that $\supp
  \cM\cap \supp\cN$ is compact.  Then \[
\sum_{i=0}^\infty (-1)^i\dim \Ext^i(\cM,\cN)=\int_{\supp
  \cM\cap \supp\cN}e(\cM)\cap e(\cN),\]
  where $\cap$ is the usual intersection product of homology classes.
\end{theorem}
More generally, one can replace pushforwards with convolutions.  If we
quantize $\fM\otimes \bar{\fM}$ (as usual, $\bar{\fM}$ is the space
$\fM$ with opposite symplectic form) with the quantization
$\cD\hat{\boxtimes}\cD^{\operatorname{op}}$, then we can view modules
over this quantization as bimodules.  In fact, the localization and
sections functors relate this category to the category of
$A\operatorname{-}A$ bimodules.  
\begin{definition}
  A $\cD\hat{\boxtimes}\cD^{\operatorname{op}}$-module $\mathcal{H}$ is {\bf
    Harish-Chandra} if it possesses a $\cQ$-lattice $\mathcal{H}^0$
  such that $\mathcal{H}^0/h \mathcal{H}^0$ is killed by $1\otimes
  f-f\otimes 1$ for any global function $f$ on $\fM$.  This is
  equivalent to $\mathcal{H}^0/h \mathcal{H}^0$ having support on the scheme $\fM\times_{\fN}\fM$.
\end{definition}
Let $\HC$ be the category of Harish-Chandra bimodules on $\cM$, and
$D^b(\HC)$ the subcategory of the bounded derived category with
Harish-Chandra cohomology.  
 Let $\nu\colon\fM\to\fN$ be the
resolution map with
Steinberg variety $\fZ:=\fM\times_{\fN}\fM$.  Consider the three different projections
$p_{ij}\colon\fM\times\fM\times\fM\to\fM\times\fM$ as well as the two
projections $p_i\colon\fM\times\fM\to\fM$.  The cohomology 
$H^{2\dim\fM}_\fZ(\fM\times\fM; \C)$ with supports in $\fZ$ is
isomorphic to $H_{2\dim\fM}^{BM}(\fZ)$ and has a convolution product
given by the formula
$$\al\star\b := (p_{13})_*(p_{12}^*\al\cdot p_{23}^*\b),$$
making it into a semisimple algebra \cite[8.9.8]{CG97}.  
For any closed subvariety $\fL\subset\fM$ with the property that $\fL =
\nu^{-1}(\nu(\fL))$, there is a degree-preserving action of this 
algebra on the cohomology $H^*_{\fL}(\M; \C)$ given by the
formula $$\al\star\gamma := (p_2)_*(\al\cdot p_1^*\gamma).$$ 

Similarly, we have a convolution product on Harish-Chandra bimodules
defined by
\begin{equation}\label{convolution}
\cH_1\star\cH_2:=(p_{13})_*(p_{12}^{-1}\cH_1\overset{L}\otimes_{p_2^{-1}\cD_{\la'}}p_{23}^{-1}\cH_2),
\end{equation}
making $D^b(\HC)$ into a monoidal category.
The same formula \eqref{convolution} defines an action of $D^b(\HC)$
on the derived category $D^b_{\fL}(\cD\mmod)$ of complexes with cohomology supported on $\fL$
\begin{theorem}
  Euler class defines an algebra homomorphism $K(\HC)\to
  H^{BM}_{\operatorname{top}}(\fM\times_{\fN}\fM)$ which interwines
  the representations on $K(D^b_{\fL}(\cD\mmod))$ and $H^{BM}_{\operatorname{top}}(\fL)$.
\end{theorem}

A number of examples of these resolutions are discussed in \cite[\S
9]{BLPWgco}.  These include:
\begin{itemize}
\item the Springer resolutions of nilcones, and the induced
  resolutions of Slodowy slices.
\item Nakajima quiver varieties \cite{Nak94,Nak98} (also discussed
  below), which are geometric avatars of the representations of Lie algebras.
\item Hypertoric varieties \cite{Pr07}, which are quaternionic
  analogues of toric varieties, and give geometric versions of various
  notions in the theory of hyperplane arrangements.
\item Lusztig slices in affine Grassmannians \cite{KWWY}.  These also
  serve to geometrize representations of Lie algebras, but in a way
  ``dual'' to Nakajima quiver varieties.  
\end{itemize}
In \cite[\S 10]{BLPWgco}, we propose a notion of duality for these
symplectic resolutions.  While many complex structures appear in the
conjectured duality, the one closest to this paper is that there is a
category $\mathcal{O}$ of special modules over $\cD\mmod$, and a dual
category $\mathcal{O}^!$ for the symplectic dual variety.
\begin{conjecture}[Braden-Licata-Proudfoot-Webster]
  There is an isomorphism $K(\mathcal{O})\cong K(\mathcal{O}^!)$ such
  that the induced actions of $K(\HC)$ and $K(\HC^!)$ commute and form
  a dual pair (in the sense of Schur-Weyl duality).  
\end{conjecture}
Schur-Weyl duality itself arises from considering this for flag
varieties in type A.  Skew-Howe duality (for type A quiver
varieties/Lusztig slices) and rank-level duality (for affine type A
quiver varieties/Lusztig slices) are special cases.

Now, let use turn to dicsussing some examples in more detail:

\subsection{Flag varieties}
\label{sec:flag-varieties}

The most familiar example is when $\fM=T^*G/B$.  In this case, we have
little to say beyond the results of Section
\ref{sec:flag-varieties-1}, since $\cD$-modules on $T^*G/B$ are just
D-modules on $G/B$ by \cite[\S 4.1]{BLPWquant}, and the categories $D^b(\HC)$
is essentially just $D^b\mathsf{Sh}_G(G/B\times G/B)$ actually
coincide, since the Steinberg $\fM\times_{\fN}\fM$ is just the union
of conormal bundles to $G$-orbits in $G/B\times G/B$.  However, this gives a
slightly different perspective, using the Springer resolution.  The cohomological
side of this picture is discussed in great detail in \cite{CG97}.  The
most important results are that:
\begin{theorem}
  We have an algebra isomorphism
  $H^{BM}_{\operatorname{top}}(\fM\times_{\fN}\fM)\cong \C[W]$.  The
  induced representations on the homology of the fibers of $\nu$ are
  the Springer representations.
\end{theorem}
Thus, any appropriate category of D-modules on $G/B$ gives a
categorified representation of the Weyl group $W$.

\subsection{Quiver varieties}
\label{sec:quiver-varieties}

Nakajima defined a remarkable set of varieties attached to finite
graphs, called {\bf quiver varieties}.  We'll leave the details of
this construction to other papers \cite{Nak94,Nak98}.  Let $I$ be the
vertex set of an oriented Dynkin diagram\footnote{We'll only consider
  the Dynkin case for simplicity, but the results carry over the
  general case with various minor caveats.}. For each pair of
dimension vectors $\Bv, \Bw$, we consider the weights $\la=\sum
w_i\omega_i,\mu=\la-\sum v_i\al_i$.  The quiver variety $\fM(\Bv,\Bw)$
(as defined in \cite[(3.5)]{Nak98})
has geometry which reflects the structure of the $\mu$-weight space of
the representation with highest weight $\la$.  It is often useful to consider
the union $\sqcup_{\Bv}\fM(\Bv,\Bw)$ which controls the whole
structure of the representation.  Nakajima also defines a Lagrangian
subvariety $\fZ\subset \sqcup_{\Bv,\Bv'}\fM(\Bv,\Bw)\times
\fM(\Bv',\Bw)$ which plays a similar role to the Steinberg.
The principal results of \cite{Nak98} show that:
\begin{theorem}\label{th:Nakajima}
  We have a surjective algebra map 
  $U(\mathfrak{g})\to H^{BM}_{\operatorname{top}}(\fZ)$.  The
  induced representations on the top homology of the fibers of $\nu$
  are the different irreducible representations with highest weight
  $\leq \la$.  
\end{theorem}

In this case, the category of Harish-Chandra bimodules is
a categorification of a quotient of the universal enveloping algebra.
While we can study it in its own terms, it is also closely
related to the categorification of these enveloping algebras as
developed by Khovanov, Lauda and Rouquier \cite{KLIII, Rou2KM}.  In
fact, it is a quotient of the categorified universal enveloping
algebra, in a certain sense.
\begin{theorem}[\mbox{\cite[3.1, 3.3 \& 3.9]{Webcatq}}]
  There is a 2-functor $\tU\to \HC$ which is ``surjective''
  (i.e. essentially surjective and full) categorifying the map of
  Theorem \ref{th:Nakajima}.
\end{theorem}
Thus, when applied to quiver varieties, we obtain a geometric avatar
of the 2-category $\tU$.  As in previous cases, this gives a natural
geometric avatar for canonical bases, this time of the whole universal
enveloping algebra \cite[3.13]{Webcatq}.

 \bibliography{../gen}
\bibliographystyle{amsalpha}
\end{document}